\documentclass[leqno,12pt]{amsart} 
\usepackage{amssymb}
\theoremstyle{plain} 
\newtheorem{theorem}{\indent\sc Theorem}[section]
\newtheorem{lemma}[theorem]{\indent\sc Lemma}

\newtheorem{proposition}[theorem]{\indent\sc Proposition}

\theoremstyle{definition} 

\begin{document}

\title[Boundedness of oscillatory integral operators and their commutators]
{Boundedness of oscillatory integral operators and their commutators
on weighted Morrey spaces}

\author[Zunwei Fu]{Zunwei Fu}

\author[Shaoguang Shi]{Shaoguang Shi$^{\sharp}$}

\author[Shanzhen Lu]{Shanzhen Lu} 


\keywords{ 
weighted Morrey space, oscillatory integral, commutator, $A_{p}$ weights.\\
\,\,\,\,$^{\sharp}$  Corresponding author: Shishaoguang@yahoo.com.cn. }

\thanks{ 
This work was partially supported by
 NSF of China (Grant Nos. 10901076, 10931001 and 11171345) and
the Key Laboratory of Mathematics and Complex System (Beijing Normal
University), Ministry of Education, China.}
\address{
School of Sciences\endgraf Linyi University \endgraf Linyi
276005\endgraf P. R. China} \email{fuzunwei@lyu.edu.cn}
\address{
School of Mathematical Sciences\endgraf Beijing Normal
University\endgraf Beijing 100875\endgraf and\endgraf School of
Sciences\endgraf Linyi University \endgraf Linyi 276005\endgraf P.
R. China} \email{shishaoguang@lyu.edu.cn}
\address{
School of Mathematical Sciences\endgraf
Beijing Normal
University\endgraf
Beijing 100875\endgraf
P. R. China}
\email{lusz@bnu.edu.cn}


\maketitle

\begin{abstract}
It is proved that both oscillatory
integral operators and fractional oscillatory
integral operators are bounded on weighted Morrey spaces. The corresponding commutators generated by $BMO$ functions are also considered.
\end{abstract}

\section{Introduction} 
To investigate the local behavior of solutions to second order elliptic partial differential equations, Morrey \cite{Mo} first introduced the classical Morrey space $M_{p,q}(\mathbb{R}^{n})$ with the norm
$$
\|f\|_{M_{p,q}(\mathbb{R}^{n})}=\sup_{B\subset \mathbb{R}^{n}}\left(\frac{1}{|B|^{1-\frac{p}{q}}}\int_{B}|f(x)|^{p}dx\right)^{\frac{1}{p}},
$$
where $f\in L_{loc}^{p}(\mathbb{R}^{n})$ and $1\leq p\leq q<\infty.$ Here and after, $B$ denotes any balls in $\mathbb{R}^{n}$.
$M_{p,q}(\mathbb{R}^{n})$ was an expansion of $L^{p}(\mathbb{R}^{n})$ in the sense that $M_{p,p}(\mathbb{R}^{n})=L^{p}(\mathbb{R}^{n})$.

In \cite{CF}, Chiarenza and Frasca obtained the boundedness of Hardy-Littlewood maximal function
$$
Mf(x)=\sup_{B\ni x}\frac{1}{|B|}\int_{B}|f(y)|dy, \eqno(1.1)
$$
on $M_{p,q}(\mathbb{R}^{n})$.

 The Calder\'{o}n-Zygmund singular integral operator is defined by
$$
\widetilde{T}f(x)=\mathrm{p.v.}\int_{\mathbb{R}^{n}}K(x-y)f(y)dy, \eqno(1.2)
$$
where $K$ is a Calder\'{o}n-Zygmund kernel(CZK).
We say a kernel $K\in C^{1}(\mathbb{R}^{n}/ \{0\})$ is a CZK if it satisfies
$$|K(x)|\leq\frac{C}{|x|^{n}},\eqno(1.3)$$

$$|\nabla K(x)|\leq\frac{C}{|x|^{n+1}} \eqno(1.4)$$
and
$$\int_{a<|x|<b}K(x)dx=0, \eqno(1.5)$$
for all $a,b$ with $0<a<b.$ Chiarenza and Frasca \cite{CF} showed the boundedness of $\widetilde{T}$ on $M_{p,q}(\mathbb{R}^{n})$.
Here and subsequently, $C$ will denote a positive constant which may vary from line to line but will remain independent of the relevant quantities.

For $0<\alpha <n$, the fractional integral operator $I_{\alpha}$ is defined by
$$
I_{\alpha}f(x)=\int_{\mathbb{R}^{n}}\frac{f(y)}{|x-y|^{n-\alpha}}dy.
$$
The boundedness of $I_{\alpha}$ on $M_{p,q}(\mathbb{R}^{n})$ was established by Adams in \cite{A}. For some works on the boundedness for the multilinear singular integral operators on Morrey type spaces, we refer to \cite{GT}, \cite{ST} and \cite{TYZ}.

In \cite{KS}, Komori and Shirai gave the definition of weighted Morrey space, which is a natural generalization of weighted Lebesgue space. Let $1\leq p<\infty,$ $0<k<1$ and $w$ be a function. Then the weighted Morrey space $M_{p,k}(w)$ was defined by
$$M_{p,k}(w)=\{f\in L_{loc}^{p}(w): \|f\|_{M_{p,k}(w)}<\infty\},$$
where
$$
\|f\|_{M_{p,k}(w)}=\sup_{B}\left(\frac{1}{w(B)^{k}}\int_{B}|f(x)|^{p}w(x)dx\right)^{\frac{1}{p}},
$$
and the supremum is taken over all balls $B\subset \mathbb{R}^{n}.$ It is obviously that if $w=1, k=1-\frac{p}{q}$, then $M_{p,k}(w)=M_{p,q}(\mathbb{R}^{n})$. For $w\in A_{p}(1\leq p<\infty)$, if $k=0,$ then $M_{p,0}(w)=L^{p}(w)$ and if $k=1, $ $M_{p,1}(w)=L^{\infty}(w)$. Also, Komori and Shirai \cite{KS} obtained the boundedness of Hardy-Littlewoood maximal operator $M$ and Calder\'{o}n-Zygmund singular integral operator $\widetilde{T}$ on $M_{p,k}(w)$ with $1<p<\infty$. When $p=1,$ the corresponding weighted weak $(1,1)$ type boundedness was also true \cite{KS}. Here $A_{p}$ and the following $A_{(p,q)}$ denote the Muckenhoupt classes \cite{Mu}
$$A_{p}:\sup_{B}\left(\frac{1}{|B|}\int_{B}w(x)dx\right)\left(\frac{1}{|B|}\int_{B}w(x)^{1-p'}dx\right)^{p-1}\leq C, 1<p<\infty$$
and
$$
A_{(p,q)}:\sup_{B}\left(\frac{1}{|B|}\int_{B}w(x)^{q}dx\right)^{\frac{1}{q}}\left(\frac{1}{|B|}\int_{B}w(x)^{-p'}dx\right)^{\frac{1}{p'}}\leq C, 1<p, q<\infty
$$
respectively. Here $1/p+1/p'=1$.

In the fractional case, we need to consider a weighted Morrey space with two weights which also introduced by Komori and Shirai \cite{KS}. Let $1\leq p<\infty$, $0<k<1$. For two weights $w_{1}$ and $w_{2}$,
$$
M_{p,k}(w_{1},w_{2})=\left\{ f:\|f\|_{M_{p,k}(w_{1},w_{2})}<\infty\right\},
$$
with the norm
$$
\|f\|_{M_{p,k}(w_{1},w_{2})}=\sup_{B}\left(\frac{1}{w_{2}(B)^{k}}\int_{B}|f(x)|^{p}w_{1}(x)dx\right)^{\frac{1}{p}},
$$
and the supremum is taken over all balls $B\subset \mathbb{R}^{n}.$ If $w_{1}=w_{2}=w$, then we denote $M_{p,k}(w_{1},w_{1})=M_{p,k}(w_{2},w_{2})=M_{p,k}(w)$. In \cite{KS}, the authors considered the weighted estimates of the fractional maximal operator $M_{\alpha}$
$$
M_{\alpha}f(x)=\sup_{B\ni x}\frac{1}{|B|^{n-\alpha}}\int_{B}|f(y)|dy
\eqno(1.6)
$$
and the fractional integral operator $I_{\alpha}$ with $0<\alpha<n$.

It is worth pointing out that the kernel in (1.2) is convolution kernel. However, there were many kinds of operators with non-convolution kernels, such as Fourier transform and Radon transform \cite{PS} which both are versions of oscillatory integrals.
The object we consider in this paper is a class of oscillatory integrals due to
Ricci and Stein  \cite{RS}
$$Tf(x)=\mathrm{p.v.}\int_{\mathbb{R}^{n}}e^{iP(x,y)}K(x-y)f(y)dy,\eqno(1.7)$$
where $P(x,y)$ is a real valued polynomial defined on $\mathbb{R}^{n}\times\mathbb{R}^{n}$, and $K$ is a CZK.

It is well known that the oscillatory factor $e^{ip(x,y)}$ makes it impossible to establish the weighted norm inequalities of (1.7) by the method as in the case of  Calder\'{o}n-Zygmund  operators or fractional integrals.
In \cite{Sa}, Sato established the weighted weak $(1,1)$ type estimate of $T$. The strong type weighted norm inequality of $T$ was proved by Lu and Zhang in \cite{LZ} with a more general case. For the other classical works about oscillatory integral, we refer to \cite{L2}, \cite{L3} and \cite{LX}.
Inspired by \cite{KS} and \cite{Sa}, we will study the weak type estimates for $T$ on weighted morrey space, that is
\begin{theorem}
If $K$ is a \emph{CZK}, $w\in A_{1}$ and $0<k<1$, then there exists a constant $C$ independent on the coefficients of $P$ such that
$$
\sup_{\lambda>0}\lambda w\left(\left\{x\in B:|Tf(x)|>\lambda\right\}\right)\leq C\|f\|_{M_{1,k}(w)}w(B)^{k}.
$$
\end{theorem}
A distribution kernel $K$ is called a standard Calder\'{o}n-Zygmund kernel(SCZK) if it satisfies the following hypotheses
$$|K(x,y)|\leq\frac{C}{|x-y|^{n}},\, x\neq y \eqno(1.8)$$
and
$$|\nabla_{x}K(x,y)|+|\nabla_{y}K(x,y)|\leq\frac{C}{|x-y|^{n+1}},\, x\neq y.\eqno(1.9)$$
The corresponding Calder\'{o}n-Zygmund integral operator $\widetilde{S}$ and oscillatory integral operator $S$ are defined by
$$\widetilde{S}f(x)=\mathrm{p.v.}\int_{\mathbb{R}^{n}}K(x,y)f(y)dy \eqno(1.10)$$
and
$$Sf(x)=\mathrm{p.v.}\int_{\mathbb{R}^{n}}e^{iP(x,y)}K(x,y)f(y)dy.\eqno(1.11)$$
where $P(x,y)$ is a real valued polynomial defined on $\mathbb{R}^{n}\times\mathbb{R}^{n}$. In 1992, Lu and zhang \cite{LZ} proved that $S$ was bounded on $L^{p}(w)$ with $1<p<\infty$, $w\in A_{p}$ by the methods of interpolation of operators with change of measures \cite{SW}.
In \cite{RS}, Ricci and Stein also introduced the standard fractional Calder\'{o}n-Zygmund kernel(SFCZK) $K_{\alpha}$ with $0<\alpha<n$, where the condition (1.8) and (1.9) were replaced by
$$|K_{\alpha}(x,y)|\leq\frac{C}{|x-y|^{n-\alpha}},\, x\neq y \eqno(1.12)$$
and
$$|\nabla_{x}K_{\alpha}(x,y)|+|\nabla_{y}K_{\alpha}(x,y)|\leq\frac{C}{|x-y|^{n+1-\alpha}},\, x\neq y.\eqno(1.13)$$
The corresponding fractional oscillatory integral operator is defined by
$$S_{\alpha}f(x)=\int_{\mathbb{R}^{n}}e^{iP(x,y)}K_{\alpha}(x,y)f(y)dy.\eqno(1.14)$$
where $P(x,y)$ is also a real valued polynomial defined on
$\mathbb{R}^{n}\times\mathbb{R}^{n}$. Obviously, when $\alpha=0$,
$S_{0}=S$ and $K_{0}=K$. Recently, the authors of this paper
obtained the weighted boundedness of $S_{\alpha}$ in \cite{SFL}.
Partly motivated by the idea from \cite{KS} and the results of
\cite{LZ} as well as \cite{SFL}, we now give another two results of
this paper
\begin{theorem} Let $1<p<\infty,$ $0<k<1$ and $w\in A_{p}$. If the Calder\'{o}n-Zygmund singular integral operator
$\widetilde{S}$ is of type $(L^{2}(\mathbb{R}^{n}), L^{2}(\mathbb{R}^{n}))$, then for any real polynomial $P(x,y)$, there exists constant $C>0$ such that
$$
\|Sf(x)\|_{M_{p,k}(w)}\leq C\|f\|_{M_{p,k}(w)},
$$
 where its norm depends only on the total degree of $P$, but not on the coefficients of $P$.
\end{theorem}

\begin{theorem} Let $\frac{1}{q}=\frac{1}{p}-\frac{\alpha}{n}$, $1<p<\frac{n}{\alpha},$ $0<k<\frac{p}{q},$ $0<\alpha<n$ and $w\in A_{(p,q)}$. For any real polynomial $P(x,y)$, there exists constant $C>0$ such that
$$
\|S_{\alpha}f(x)\|_{M_{q,\frac{qk}{p}}(w^{q})}\leq C\|f\|_{M_{p,k}(w^{p},w^{q})},
$$
where its norm depends only on the total degree of $P$, but not on the coefficients of $P$.
If $p=1$ and $w\in A_{(1,q)}$ with $q=\frac{n}{n-\alpha}$, then for all $\lambda>0,$ there exists constant $C>0$ such that
$$
 w^{q}\left(\left\{x\in B:|S_{\alpha}f(x)|>\lambda\right\}\right)\leq \frac{C}{\lambda^{q}}\|f\|_{M_{1,k}(w,w^{q})}^{q}(w^{q}(B))^{kq}.
$$
\end{theorem}

We prove Theorem 1.1-Theorem 1.3 in Section 2.  In Section 3, we set up the weighted norm inequalities for the corresponding commutators of $S$ and $S_{\alpha}$, respectively.
Throughout this paper all definitions and notations are standard. A weight $w$ is a locally integrable function on $\mathbb{R}^{n}$ which takes values in $(0,\infty)$ almost everywhere.
$B=B(x_{0},r)$ denotes the ball with center $x_{0}$ and radius $r.$ Given $\lambda>0,$ $\lambda B=B(x_{0},\lambda r)$.

\section{Weighted estimates for oscillatory integral operators}

We begin this section with some properties of $A_{p}$ weight classes which play important role in the proofs of our main results.
\begin{lemma}\label{Lemma} \,\,\cite{G}
Let $1\leq p<\infty$, and $w\in A_{p}$. Then the following statements are true

$\mathrm{(1)}$\,\, There exists a constant $C$ such that $$w(2B)\leq Cw(B).\eqno(2.1)$$
When $w$ satisfies this condition, we say $w$ satisfies doubling condition.

$\mathrm{(2)}$\,\,  There exists a constant $C>1$ such that $$w(2B)\geq Cw(B).\eqno(2.2)$$
When $w$ satisfies this condition, we say $w$ satisfies reverse doubling condition.

$\mathrm{(3)}$\,\,  There exist two constant $C$ and $r>1$ such that the following reverse H\"{o}lder inequality holds for every ball $B\subset \mathbb{R}^{n}$
$$\left(\frac{1}{|B|}\int_{B}w(x)^{r}dx \right)^{\frac{1}{r}}\leq C\left(\frac{1}{|B|}\int_{B} w(x)dx\right).\eqno(2.3)$$

$\mathrm{(4)}$\,\,  For all $\lambda>1,$ we have
$$
w(\lambda B)\leq C\lambda^{np}w(B).\eqno(2.4)
$$

$\mathrm{(5)}$\,\,  There exist two constant $C$ and $\delta>0$ such that for any measurable set $Q\subset B$
$$\frac{w(Q)}{w(B)}\leq C\left( \frac{|Q|}{|B|}\right)^{\delta}.\eqno(2.5)$$
If $w$ satisfies $(2.5)$, we say $w\in A_{\infty}$.

$\mathrm{(6)}$\,\,  For all $\frac{1}{p}+\frac{1}{p'}=1$ we have
$$
A_{\infty}=\bigcup_{1<p<\infty}A_{p}, \,\,\,\,\,\, w^{1-p'}\in A^{p'}.\eqno(2.6)
$$
\end{lemma}
Our argument based heavily on the following well-known results about $T$, $S$ and $S_{\alpha}$.

\begin{lemma} \,\,\cite{Sa} If $K$ is a \emph{CZK}, $w\in A_{1}$, then there exist constant $C>0$ independent on the coefficients of $P$ such that
$$
\sup_{\lambda>0}\lambda w\left(\left\{x\in \mathbb{R}^{n}:|Tf(x)|>\lambda\right\}\right)\leq C\|f\|_{L^{1}(w)}.
$$
\end{lemma}
\begin{lemma} \,\,\cite{LZ} If $K$ is a \emph{SCZK}, $w\in A_{p}(1<p<\infty)$ and the Calder\'{o}n-Zygmund singular integral operator
$\widetilde{S}$
is of type $(L^{2}(\mathbb{R}^{n}), L^{2}(\mathbb{R}^{n}))$, then for any real polynomial $P(x,y)$, there exists $C>0$ such that
$$\|Sf\|_{L^{p}(w)}\leq C\|f\|_{L^{p}(w)},$$
where its norm depends only on the total degree of $P$, but not on the coefficients of $P$.
\end{lemma}
\begin{lemma} \,\,\cite{SFL}, \cite{KS} Let $w\in A_{(p,q)}$ , $0<\alpha<n$, $1< p<\frac{n}{\alpha}$ and $\frac{1}{p}-\frac{1}{q}=\frac{\alpha}{n}$. Then there exists constant $C>0$ independent of $P$ such that
such that
$$\|S_{\alpha}f\|_{L^{q}(w^{q})}\leq C\|f\|_{L^{p}(w^{p})},$$
and
$$
\|I_{\alpha}f(x)\|_{M_{q,\frac{qk}{p}}(w^{q})}\leq C\|f\|_{M_{p,k}(w^{p},w^{q})}.
$$
When $p=1$ and $w\in A_{(1,q)}$ with $q=\frac{n}{n-\alpha}$, then for all $\lambda>0,$ there exists constant $C>0$ such that
$$
 w^{q}\left(\left\{x\in B:|I_{\alpha}f(x)|>\lambda\right\}\right)\leq \frac{C}{\lambda^{q}}\|f\|_{M_{1,k}(w,w^{q})}^{q}(w^{q}(B))^{kq}.
$$
\end{lemma}

We first give the proof of Theorem 1.1.
 Decompose $f=f\chi_{2B}+f_{\chi_{(2B)^{c}}}:=f_{1}+f_{2}.$ For any given $\lambda>0$, we write
\begin{align*}
w\left(\{x\in B: |Tf(x)|>\lambda\}\right)&\leq w\left(\left\{x\in B: |Tf_{1}(x)|>\frac{\lambda}{2}\right\}\right)+w\left(\left\{x\in B: |Tf_{2}(x)|>\frac{\lambda}{2}\right\}\right)\\
&:= I+II.
\end{align*}
An application of (2.1) and Lemma 2.2 yields that
$$
I\leq \left(w\left\{x\in \mathbb{R}^{n}: |Tf_{1}(x)|>\frac{\lambda}{2}\right\}\right)\leq \frac{C}{\lambda}\int_{2B}|f(x)|w(x)dx\leq \frac{C}{\lambda}\|f\|_{M_{1,k}}(w)w(B)^{k}.\eqno(2.7)
$$

Next we turn to deal with the term $II.$ An elementary estimate shows
\begin{align*}
w\left(\left\{x\in B: |Tf_{2}(x)|>\frac{\lambda}{2}\right\}\right)&=\int_{\left\{x\in B: |Tf_{2}(x)|>\frac{\lambda}{2}\right\}}w(x)dx\\
 &\leq \frac{C}{\lambda}\int_{\left\{x\in B: |Tf(x)|>\frac{\lambda}{2}\right\}}|Tf_{2}(x)|w(x)dx.
\end{align*}
We note that for $x\in B$ and $y\in (2B)^{c}$, $|x_{0}-y|<C|x-y|$. Applying $(1.3)$, we conclude that
\begin{align*}
|Tf_{2}(x)|&=|\int_{(2B)^{c}}e^{iP(x,y)}K(x-y)f_{2}(y)dy|\\
&\leq C\int_{\mathbb{R}^{n}}\frac{|f_{2}(y)|}{|x-y|^{n}}dy\\
&\leq C\int_{|x_{0}-y|>2r}\frac{|f_{2}(y)|}{|x_{0}-y|^{n}}dy\\
&\leq C\sum_{j=1}^{\infty}\int_{2^{j}r<|x_{0}-y|<2^{j+1}r}\frac{|f(y)|}{|x_{0}-y|^{n}}dy\\
&\leq C\sum_{j=1}^{\infty}\frac{1}{|2^{j}B|}\int_{2^{j+1}B}|f(y)|dy.
\end{align*}
H\"{o}lder inequality and the $A_{p}$ condition imply that
$$
II\leq \frac{C}{\lambda}\sum_{j=1}^{\infty}\frac{1}{|2^{j}B|}\int_{2^{j+1}B}|f(y)|dy\int_{\left\{x\in B: |Tf(x)|>\frac{\lambda}{2}\right\}}w(x)dx
\leq \frac{C}{\lambda}\|f\|_{M_{1,k}}(w)w(B)^{k}.\eqno(2.8)
$$
Then, Theorem 1.1 is a by-product of (2.7)-(2.8).

We now give the proof of Theorem 1.2.
As in \cite{KS}, our method is adapted from \cite{FLY} in the case of Lebesgue measure. Let $1<p<\infty, 0<k<1$, $K$ be a SCZK. It suffices to show that
$$
\frac{1}{w(B)^{k}}\int_{B}|Sf(x)|^{p}w(x)dx\leq C\|f\|_{M_{p,k}(w)}^{p}.\eqno(2.9)
$$
For a fixed ball $B=B(x_{0},r)$, we decompose $f=f\chi_{2B}+f_{\chi_{(2B)^{c}}}:=f_{1}+f_{2}$ and consider the corresponding splitting
\begin{align*}
Sf(x)&=\int_{2B}e^{iP(x,y)}K(x,y)f(y)dy+\int_{(2B)^{c}}e^{iP(x,y)}K(x,y)f(y)dy\\
&=:Sf_{1}(x)+Sf_{2}(x).
\end{align*}

Since $S$ is a linear operator, so we get
$$
\frac{1}{w(B)^{k}}\int_{B}|Sf(x)|^{p}w(x)dx\leq \frac{1}{w(B)^{k}}\int_{B}(|Sf_{1}(x)|^{p}+|Sf_{2}(x)|^{p})w(x)dx:=I+II.\eqno(2.10)
$$
It follows from Lemma 2.3 and (2.1) that
$$
I\leq \frac{1}{w(B)^{k}}\int_{\mathbb{R}^{n}}|Sf_{1}(x)|^{p}w(x)dx\leq \frac{C}{w(B)^{k}}\int_{2B}|f(x)|^{p}w(x)dx
\leq C\|f\|_{M_{p,k}(w)}^{p}.\eqno(2.11)
$$

We are now in a position to estimate the term $II$. We note that for $x\in B$ and $y\in (2B)^{c}$, $|x_{0}-y|<C|x-y|$. Applying $(1.8)$, we conclude that
\begin{align*}
|Sf_{2}(x)|&=|\int_{(2B)^{c}}e^{iP(x,y)}K(x,y)f_{2}(y)dy|\\
&\leq C\int_{\mathbb{R}^{n}}\frac{|f_{2}(y)|}{|x-y|^{n}}dy\\
&\leq C\int_{|x_{0}-y|>2r}\frac{|f_{2}(y)|}{|x_{0}-y|^{n}}dy\\
&\leq C\sum_{j=1}^{\infty}\int_{2^{j}r<|x_{0}-y|<2^{j+1}r}\frac{|f(y)|}{|x_{0}-y|^{n}}dy\\
&\leq C\sum_{j=1}^{\infty}\frac{1}{|2^{j}B|}\int_{2^{j+1}B}|f(y)|dy.
\end{align*}
H\"{o}lder inequality and the $A_{p}$ condition imply that
\begin{align*}
\int_{2^{j+1}B}|f(y)|dy&\leq\left(\int_{2^{j+1}B}|f(y)|^{p}w(y)dy\right)^{\frac{1}{p}}\left(\int_{2^{j+1}B}w(y)^{-\frac{p'}{p}}dy\right)^{\frac{1}{p'}}\\
&\leq C\|f\|_{M_{p,k}(w)}w(2^{j+1}B)^{\frac{k}{p}}\left(\int_{2^{j+1}B}w(y)^{1-p'}dy\right)^{1-\frac{1}{p'}}\\
&\leq C\|f\|_{M_{p,k}(w)}|2^{j+1}B|w(2^{j+1}B)^{\frac{1}{p}(k-1)}.
\end{align*}
Then, using (2.2) we obtain
$$
II\leq  C\|f\|_{M_{p,k}(w)}^{p}\left(\sum_{j=1}^{\infty}\frac{w(B)^{\frac{1-k}{p}}}{w(2^{j+1}B)^{\frac{1-k}{p}}}\right)^{p}
\leq C\|f\|_{M_{p,k}(w)}^{p}.\eqno(2.12)
$$
We get (2.9) by (2.10), (2.11) and (2.12).
The proof of Theorem 1.2 is completed.

For the last part of this section, we show the proof of Theorem 1.3.
 Theorem 1.3 is a byproduct of Lemma 2.4 and the following observation
$$
\left| S_{\alpha}f(x)\right|\leq \int\frac{|f(y)|}{|x-y|^{n-\alpha}}dy=I_{\alpha}(|f|)(x).
$$
In fact, if $1<p<\infty$,
\begin{eqnarray*}
\frac{1}{w(B)^{\frac{qk}{p}}}\int|S_{\alpha}f(x)|^{q}w^{q}
&\leq& \frac{1}{w(B)^{\frac{qk}{p}}}\int|I_{\alpha}f(x)|^{q}w^{q}\\
&\leq& C\|f\|_{M_{p,k}(w^{p},w_{q})}.
\end{eqnarray*}

In the same manner, we can obtain the result for $p=1.$

\section{Weighted estimates for the commutators}

The aim of this section is to set up the weighted boundedness for the commutators formed by $S(S_{\alpha})$ and $BMO(\mathbb{R}^{n})$ functions.

A locally integrable function $b$ is said to be in $BMO(\mathbb{R}^{n})$ if for any ball $B\subset \mathbb{R}^{n}$
$$\|b\|_{BMO(\mathbb{R}^{n})}=\sup_{B}\frac{1}{|B|}\int_{B}|b(x)-b_{B}|dx<\infty,$$
where $b_{B}=\frac{1}{|B|}\int_{B}b(x)dx.$

We next formulate some remarks about $BMO(\mathbb{R}^{n}).$
\begin{lemma}\,\,\,\, \cite{L1},\cite{T}
Let $1\leq p<\infty$, $b\in BMO(\mathbb{R}^{n})$. Then for any ball $B\subset \mathbb{R}^{n}$, the following statements are true

$\mathrm{(1)}$\,\, There exist constants $C_{1}$, $C_{2}$ such that for all $\alpha>0$
$$\left|\{x\in B:|b(x)-b_{B}|>\alpha\}\right|\leq C_{1}|B|e^{-C_{2}\alpha/\|b\|_{BMO(\mathbb{R}^{n})}}.\eqno(3.1)$$
The inequality $(3.1)$ is also called John-Nirenberg inequality.

$\mathrm{(2)}$\,\,
$$|b_{2^{\lambda}B}-b_{B}|\leq2^{n}\lambda \|b\|_{BMO(\mathbb{R}^{n})}.\eqno(3.2)$$
\end{lemma}
\begin{lemma} \,\,\,\, \cite{MW} Let $w\in A_{\infty}$. Then the following statements are equivalent

$\mathrm{(1)}$\,\,
$$\|b\|_{BMO(\mathbb{R}^{n})}\sim \sup_{B}\left(\frac{1}{|B|}\int_{B}|b(x)-b_{B}|^{p}dx\right)^{\frac{1}{p}}.\eqno(3.3)$$

$\mathrm{(2)}$\,\,
$$\|b\|_{BMO(\mathbb{R}^{n})}\sim \sup_{B}\inf_{a\in \mathbb{R}}\frac{1}{|B|}\int_{B}|b(x)-a|dx.\eqno(3.4)$$

$\mathrm{(3)}$\,\,
$$
\|b\|_{BMO(w)}=\sup_{B}\frac{1}{w(B)}\int_{B}|b(x)-b_{B,w}|w(x)dx.\eqno(3.5)
$$
where $BMO(w)=\{b:\|b\|_{BMO(w)}<\infty\}$ and $b_{B,w}=\frac{1}{w(B)}\int_{B}b(y)w(y)dy.$
\end{lemma}
For a locally integrable function $b$, the commutator formed by $S(S_{\alpha})$ and $b$ are defined by
$$
S_{b}:=[b, S]f(x)=b(x)Sf(x)-S(bf)(x)
$$
and
$$
S_{\alpha, b}:=[b, S_{\alpha}]f(x)=b(x)S_{\alpha}f(x)-S_{\alpha}(bf)(x).
$$
Our main results of this section are
\begin{theorem} Let $p, k, w, K$ and $\widetilde{S}$ be the same as Theorem $1.2$. If $b\in BMO(\mathbb{R}^{n})$, then for any real polynomial $P(x,y)$, there exists constant $C>0$ such that
$$
\|S_{b}f(x)\|_{M_{p,k}(w)}\leq C\|f\|_{M_{p,k}(w)},
$$
where its norm depends only on the total degree of $P$, but not on the coefficients of $P$.
\end{theorem}

\begin{theorem} Let $p, q,$ $k,$ $K_{\alpha}$ and $w$ be the same as in Theorem $1.3$. Then for $b\in BMO(\mathbb{R}^{n})$, there exists constant $C>0$ such that
$$
\|S_{\alpha,b}f(x)\|_{M_{q,\frac{qk}{p}}(w^{q})}\leq C\|f\|_{M_{p,k}(w^{p},w^{q})}.
$$
\end{theorem}
The following results will play an important role in our analysis.
\begin{lemma} \,\,\,\,\cite{Sh}     Suppose $K$ is a \emph{SCZK}, $w\in A_{p}(1<p<\infty)$ and the operator
$\widetilde{S}$ is of type $(L^{2}(\mathbb{R}^{n}), L^{2}(\mathbb{R}^{n}))$. Then for any $b\in BMO(\mathbb{R}^{n})$, there exists constants $C>0$ independent on the coefficients of $P$ such that
$$\|S_{b}f\|_{L^{p}(w)}\leq C\|f\|_{L^{p}(w)}.$$
\end{lemma}
\begin{lemma} \,\,\,\,\cite{KS} Let $1<p<\frac{n}{\alpha},$ $0<k< \frac{p}{q},$ $0<\alpha<n$, $\frac{1}{q}=\frac{1}{p}-\frac{\alpha}{n}$ and $w\in A_{(p,q)}$. Then for $b\in BMO(\mathbb{R}^{n})$, there exists constant $C>0$ such that
$$
\|I_{\alpha,b}f(x)\|_{M_{q,\frac{qk}{p}}(w^{q})}\leq C\|f\|_{M_{p,k}(w^{p},w^{q})}.
$$
\end{lemma}
Lemma 3.6 is the weighted version of Theorem 1 in \cite{KM}.

The following Proposition is essential to the proof of Theorem 3.3.
\begin{proposition}
Let $B=B(x_{0},r)$, $0<k<1$ and $1<p<\infty$. Then the inequality
$$\left(\int_{|x_{0}-y|>2r}\frac{|f(y)|}{|x_{0}-y|^{n}}|b_{B,w}-b(y)|dy\right)^{p}w(B)^{1-k}\leq C\|f\|_{M_{p,k}(w)}^{p}\|b\|_{BMO(\mathbb{R}^{n})}^{p}.\eqno(3.6)$$
holds for every $y\in (2B)^{c}$, where $(2B)^{c}=\mathbb{R}^{n}/2B.$
\end{proposition}
\begin{proof}
Using H\"{o}lder's inequality to the left-hand-side of (3.6), we have
\begin{align*}
&\left(\int_{|x_{0}-y|>2r}\frac{|f(y)|}{|x_{0}-y|^{n}}|b_{B,w}-b(y)|dy\right)^{p}w(B)^{1-k}\\
&\leq \left(\sum_{j=1}^{\infty}\int_{2^{j}r<|x_{0}-y|<2^{j+1}r}\frac{|f(y)|}{|x_{0}-y|^{n}}|b_{B,w}-b(y)|dy\right)^{p}w(B)^{1-k}\\
&\leq \left(\sum_{j=1}^{\infty}\frac{1}{|2^{j}B|}\int_{2^{j+1}B}|f(y)||b_{B,w}-b(y)|dy\right)^{p}w(B)^{1-k}\\
&\leq C\left[\sum_{j=1}^{\infty}\frac{1}{|2^{j}B|}\left(\int_{2^{j+1}B}|f(y)|^{p}w(y)dy\right)^{\frac{1}{p}}\left(\int_{2^{j+1}B}|b_{B,w}-b(y)|^{p'}w(y)^{1-p'}dy\right)^{\frac{1}{p'}}\right]^{p}w(B)^{1-k}\\
&\leq C\|f\|_{M_{p,k}(w)}^{p}\left[\sum_{j=1}^{\infty}\frac{w(2^{j+1}B)^{\frac{k}{p}}}{|2^{j}B|}\left(\int_{2^{j+1}B}|b_{B,w}-b(y)|^{p'}w(y)^{1-p'}dy\right)^{\frac{1}{p'}}\right]^{p}w(B)^{1-k}.
\end{align*}

For the simplicity of analysis, we denote $A$ as
$$
\left(\int_{2^{j+1}B}|b_{B,w}-b(y)|^{p'}w(y)^{1-p'}dy\right)^{\frac{1}{p'}}.
$$
By an elementary estimate, we have
\begin{align*}
A
&\leq \left(\int_{2^{j+1}B}(|b_{2^{j+1}B,w^{1-p'}}-b(y)|+|b_{2^{j+1}B,w^{1-p'}}-b_{B,w}|)^{p'}w(y)^{1-p'}dy\right)^{\frac{1}{p'}}\\
&\leq \left\|\frac{|b_{2^{j+1}B,w^{1-p'}}-b(\cdot)|+|b_{2^{j+1}B,w^{1-p'}}-b_{B,w}|}{w(\cdot)}\right\|_{L^{p'}(w)}\\
&\leq \left(\int_{2^{j+1}B}|b_{2^{j+1}B,w^{1-p'}}-b(y)|w(y)^{1-p'}dy\right)^{\frac{1}{p'}}+|b_{2^{j+1}B,w^{1-p'}}-b_{B,w}|w^{1-p'}(2^{j+1}B)^{\frac{1}{p'}}\\
&= :J+JJ.
\end{align*}

For the term $J$, Lemma 3.2 implies
$$
J\leq C\|b\|_{BMO(w^{1-p'})}w^{1-p'}(2^{j+1}B)^{\frac{1}{p'}}\leq Cw^{1-p'}(2^{j+1}B)^{\frac{1}{p'}}.\eqno(3.7)
$$
To deal with $JJ$, by (3.2), we have
\begin{equation*}
\begin{split}
|b_{2^{j+1}B,w^{1-p'}}-b_{B,w}|
&\leq |b_{2^{j+1}B,w^{1-p'}}-b_{2^{j+1}B}|+|b_{2^{j+1}B}-b_{B}|+|b_{B}-b_{B,w}|\\
&\leq \frac{1}{w^{1-p'}(2^{j+1}B)}\int_{2^{j+1}B}|b(y)-b_{2^{j+1}B}|w(y)^{1-p'}dy+2^{n}(j+1)\|b\|_{BMO(\mathbb{R}^{n})}\\
&\quad\quad\quad\quad+\frac{1}{w(B)}\int_{B}|b(y)-b_{B}|w(y)dy\\
&:=JJ_{1}+JJ_{2}+JJ_{3}.
\end{split}
\end{equation*}
Combining (2.5) with (3.1),
\begin{align*}
JJ_{3}
&= \frac{1}{w(B)}\int_{0}^{\infty}w(\{x\in B:|b(y)-b_{B}|>\alpha \})d\alpha\\
&\leq C\int_{0}^{\infty}e^{-C_{2}\alpha\delta/\|b\|_{BMO(\mathbb{R}^{n})}}d\alpha\\
&\leq C.
\end{align*}
In the same manner we can see that
$$
JJ_{1}\leq C.
$$
It follows immediately that
$$
JJ\leq C(2^{n}(j+1)+2)w^{1-p'}(2^{j+1}B)^{\frac{1}{p'}}.\eqno(3.8)
$$
As a by-product of (3.7) and (3.8), we have
$$
A\leq C(j+1)w^{1-p'}(2^{j+1}B)^{\frac{1}{p'}}.
$$
Then, applying (2.2), the proof of (3.6) based on the following observation
\begin{align*}
&\left[\sum_{j=1}^{\infty}\frac{w(2^{j+1}B)^{\frac{k}{p}}}{|2^{j}B|}\left(\int_{2^{j+1}B}|b(y)-b_{B,w}|^{p'}w(y)^{1-p'}dy\right)^{\frac{1}{p'}}\right]^{p}
w(B)^{1-k}\\
&\leq C\left[\sum_{j=1}^{\infty}\frac{w(B)^{\frac{1-k}{p}(j+1)}}{w(2^{j+1}B)^{\frac{1-k}{p}}}\right]^{p}=C.
\end{align*}
\end{proof}

{\it Proof of Theorem $3.3$.}
The task is now to find a constant $C$ such that for fixed ball $B=B(x_{0},r)$, we can obtain
$$
\frac{1}{w(B)^{k}}\int_{B}\left|S_{b}f(x)\right|^{p}w(x)dx\leq C\|f\|_{M_{p,k}(w)}^{p}.\eqno(3.9)
$$

We decompose $f=f\chi_{2B}+f_{\chi_{(2B)^{c}}}:=f_{1}+f_{2},$ and consider the corresponding splitting
\begin{align*}
\int_{B}\left|S_{b}f(x)\right|^{p}w(x)dx&\leq C\left(\int_{B}|S_{b}f_{1}(x)|^{p}w(x)dx+\int_{B}|S_{b}f_{2}(x)|^{p}w(x)dx\right)\\
&=:K+KK.
\end{align*}

An application of Lemma 3.5 and $w\in A_{p}$ yields
$$
K\leq C\int_{2B}|f(x)|^{p}w(x)dx\leq C\|f\|_{M_{p,k}(w)}^{p}w(B)^{k}.\eqno(3.10)
$$

To estimate the other term $KK$, we note that for $x\in B$ and $y\in (2B)^{c}$, $|x_{0}-y|<C|x-y|$. Then a further use of $(1.8)$ derives that
\begin{align*}
\left|S_{b}f_{2}(x)\right|^{p}&=\left|\int_{\mathbb{R}^{n}}e^{ip(x,y)}K(x,y)f_{2}(y)(b(x)-b(y))dy\right|^{p}\\
&\leq C\left(\int_{\mathbb{R}^{n}}\frac{|f_{2}(y)||b(x)-b(y)|}{|x-y|^{n}}dy\right)^{p}\\
&\leq C\left(\int_{|x_{0}-y|>2r}\frac{|f(y)|}{|x_{0}-y|^{n}}\{|b(x)-b_{B,w}|+|b_{B,w}-b(y)|\}dy\right)^{p}.
\end{align*}
where $b_{B,w}=\frac{1}{w(B)}\int_{B}b(x)w(x)dx$. Then, we have
\begin{align*}
KK&\leq C\left(\int_{|x_{0}-y|>2r}\frac{|f(y)|}{|x_{0}-y|^{n}}dy\right)^{p}\int_{B}|b(x)-b_{B,w}|^{p}w(x)dx\\
&\quad\quad\quad\quad+C\left(\int_{|x_{0}-y|>2r}\frac{|f(y)|}{|x_{0}-y|^{n}}|b(y)-b_{B,w}|dy\right)^{p}w(B)\\
&:= KK_{1}+KK_{2}.
\end{align*}
A further use of proposition 3.7, we get
$$KK_{2}\leq C\|f\|_{M_{p,k}(w)}^{p}w(B)^{k}.$$

To get the desired estimate, we are led to estimate the term $KK_{1}$. This estimate will be done via (2.1), (2.3) and Lemma 3.2.

\begin{equation*}
\begin{split}
KK_{1}&=\left(\sum_{j=1}^{\infty}\int_{2^{j}r<|x_{0}-y|<2^{j+1}r}\frac{|f(y)|}{|x_{0}-y|^{n}}dy\right)^{p}\int_{B}|b(x)-b_{B,w}|^{p}w(x)dx\\
&\leq\left(\sum_{j=1}^{\infty}\frac{1}{|2^{j}B|}\int_{2^{j+1}B}|f(y)|dy\right)^{p}\int_{B}|b(x)-b_{B,w}|^{p}w(x)dx\\
&\leq C\sum_{j=1}^{\infty}\frac{1}{|2^{j}B|}\left(\frac{1}{w(2^{j+1}B)^{k}}\int_{2^{j+1}B}|f(y)|^{p}w(y)dy\right)^{\frac{1}{p}}\\
&\quad\quad\quad\quad\times w(2^{j+1}B)^{\frac{k}{p}}\left(\int_{2^{j+1}B}w(y)^{-\frac{1}{p-1}}dy\right)^{\frac{p-1}{p}}\int_{B}|b(x)-b_{B,w}|^{p}w(x)dx\\
&\leq C\|f\|_{M_{p,k}(w)}\left(\sum_{j=1}^{\infty}\frac{|2^{j+1}B|^{-\frac{1}{p}}}{|2^{j}B|}\left(\frac{1}{|2^{j+1}B|}\int_{2^{j+1}B}w(y)dy\right)^{-\frac{1}{p}}w(2^{j+1}B)^{\frac{k}{p}}\right)^{p}\\
&\quad\quad\quad\quad\times \int_{B}|b(x)-b_{B,w}|^{p}w(x)dx\\
&\leq C\|f\|_{M_{p,k}(w)}^{p}\|b\|_{BMO(\mathbb{R}^{n})}^{p}\sum_{j=1}^{\infty}\left(\frac{w(B)^{\frac{1-k}{p}}}{w(2^{j+1}B)^{\frac{1-k}{p}}}\right)^{p}w(B)^{k}\\
&\leq C\|f\|_{M_{p,k}(w)}^{p}w(B)^{k}
\end{split}
\end{equation*}
Hence
$$KK\leq C\|f\|_{L^{p,k}(w)}^{p}w(B)^{k}.\eqno(3.11)$$

Combing (3.10), (3.11), we obtain (3.9), which is the desired conclusion.

{\it Proof of Theorem $3.4$.}\,\,  As in the proof of Theorem 1.3.
Theorem 3.4 can be deduced via Lemma 3.6 and the following
observation
$$
\left| S_{b}f(x)\right|\leq \int\frac{|f(y)|}{|x-y|^{n-\beta}}dy=I_{\alpha}(|f|)(x).
$$


\end{document}